\documentclass[leqno]{amsart}
\usepackage{amsmath}
\usepackage{mathtools}
\usepackage{amssymb}
\usepackage{amsthm}
\usepackage{graphicx}
\usepackage{enumerate}
\usepackage[mathscr]{eucal}
\theoremstyle{plain}
\newtheorem{theorem}{Theorem}[section]
\newtheorem{prop}[theorem]{Proposition}
\newtheorem{cor}{Corollary}[theorem]

\theoremstyle{definition}
\newtheorem{definition}{Definition}[section]
\newtheorem{remark}{Remark}[section]

\newtheorem{example}{Example}[theorem]

\numberwithin{equation}{section}
\usepackage[pagewise]{lineno}

\begin{document}

\title[Smoothness and norm attainment of bilinear operators]{Smoothness and norm attainment of bounded bilinear operators between Banach spaces}
\author[Sain]{Debmalya Sain}

\newcommand{\acr}{\newline\indent}

\address[Sain]{Department of Mathematics\\ Indian Institute of Science\\ Bengaluru 560012\\ Karnataka \\India\\ }
\email{saindebmalya@gmail.com}


\thanks{The research of the author is sponsored by Dr. D. S. Kothari Postdoctoral Fellowship under the mentorship of Professor Gadadhar Misra. The author feels elated to acknowledge the delightful and inspiring company of his beloved fellow researcher Dr. Chandrodoy Chattopadhyay.}


\begin{abstract}
We study the smoothness and the norm attainment of bounded bilinear operators between Banach spaces, using the concepts of Birkhoff-James orthogonality and semi-inner-products. In the finite-dimensional case, we characterize Birkhoff-James orthogonality of bilinear operators in terms of the norm attainment set. This yields a nice characterization of smoothness of bilinear operators between Banach spaces, in the finite-dimensional case. Without any restriction on the dimension of the space, we obtain a complete characterization of the norm attainment set of a bounded bilinear operator using semi-inner-products, which is particularly useful when the concerned Banach spaces are smooth.  
\end{abstract}

\maketitle

\section{Introduction.}

In recent times, geometry of bounded linear operators between Banach spaces has been a topic of considerable interest \cite{BS,K,SP,S,Sa,T}. A common theme in each of these articles is Birkhoff-James orthogonality \cite{B,J} in Banach spaces and its interaction with bounded linear operators. Very recently, the norm attainment problem for symmetric multilinear forms on Hilbert spaces has been explored in \cite{CR}. The purpose of the present article is to explore the smoothness and the norm attainment of bounded bilinear operators between Banach spaces.\\

In this article, letters $ \mathbb{X}, \mathbb{Y}, \mathbb{Z}$ stand for Banach spaces. We will consider only real Banach spaces of dimension strictly greater than $ 1 $ throughout the article, without mentioning it any further. Since there is no chance of confusion, we will use the same symbol $ \|.\| $ to denote the norm function on any Banach space. Let $ B_{\mathbb{X}} = \{ x \in \mathbb{X} : \| x \| \leq 1 \} $ and $ S_{\mathbb{X}} = \{ x \in \mathbb{X} : \| x \| = 1 \} $ denote the unit ball and the unit sphere of $ \mathbb{X} $ respectively. Given any $ x \in \mathbb{X} $ and any $ r > 0, $ let $ B(x,r) = \{ y \in \mathbb{X} : \| y-x \| < r \} $ denote the open ball with radius $ r $ and center at $ x. $ It is easy to see that $ \mathbb{X} \times \mathbb{Y}, $ equipped with the norm $ \| (x,y) \| = \max \{ \|x\|,\|y\| \} $ for all $ (x,y) \in \mathbb{X} \times \mathbb{Y}, $ is a Banach space. Throughout the present article, we consider $ \mathbb{X} \times \mathbb{Y} $ to be endowed with the aforesaid norm. Let $ \mathcal{T}: \mathbb{X} \times \mathbb{Y} \longrightarrow \mathbb{Z} $ be a bilinear operator, i.e., $ \mathcal{T} $ is linear in each argument. The norm of $ \mathcal{T} $ is defined as $ \| \mathcal{T} \| = \sup \{ \| \mathcal{T}(x,y) \| :  (x,y) \in S_{\mathbb{X} \times \mathbb{Y} } \}. $ Let $ M_{\mathcal{T}} $ denote the norm attainment set of $ \mathcal{T}, $ i.e., $ M_{\mathcal{T}} = \{ (x,y) \in S_{\mathbb{X} \times \mathbb{Y}} : \| \mathcal{T}(x,y) \| = \| \mathcal{T} \| \}. $ We say that $ \mathcal{T} $ is bounded if $ \| \mathcal{T} \| < \infty .$ The Banach space of all bounded bilinear operators from $ \mathbb{X} \times \mathbb{Y} $ to $ \mathbb{Z}, $ equipped with the above mentioned norm is denoted by $ \mathscr{B}(\mathbb{X} \times \mathbb{Y},\mathbb{Z}). $\\

Our first major result is to describe the smooth points of $ \mathscr{B}(\mathbb{X} \times \mathbb{Y},\mathbb{Z}), $ in the finite-dimensional case. This follows from a characterization of Birkhoff-James orthogonality of bounded bilinear operators between finite-dimensional Banach spaces. We next study the norm attainment set of a bounded bilinear operator between Banach spaces and contrast it with the corresponding norm attainment set of a bounded linear operator between Banach spaces. Finally, we obtain a complete characterization of the norm attainment set of a bounded bilinear operator between Banach spaces. In order to fulfill our goal, we apply the concept of semi-inner-products in Banach spaces \cite{G,L} and Birkhoff-James orthogonality techniques. Let us briefly recall the relevant definitions in this context.  

\begin{definition}
Let $ \mathbb{X} $ be a Banach space and let $ x,y \in \mathbb{X}.  $ We say that $ x $ is Birkhoff-James orthogonal to $ y, $ written as $ x \perp_B y, $ if $ \| x + \lambda y \| \geq \| x \| $ for all scalars $ \lambda. $ 
\end{definition}

Birkhoff-James orthogonality was decomposed into two parts in \cite{S} by introducing the positive part and the negative part of an element in a Banach space.

\begin{definition}
Let $ \mathbb{X} $ be a Banach space and let $ x \in \mathbb{X}.  $ The positive part of $ x $ is defined as $ x^{+} = \{ y \in \mathbb{X} : \| x+ \lambda y \| \geq \| x \| ~\textrm{for all}~ \lambda \geq 0 \}. $ Similarly, the negative part of $ x $ is defined as $ x^{-} = \{ y \in \mathbb{X} : \| x+ \lambda y \| \geq \| x \| ~\textrm{for all}~ \lambda \leq 0 \}. $ The Birkhoff-James orthogonality set of $ x $ is defined as $ x^{\perp} = \{ y \in \mathbb{X} : x \perp_B y\}. $ 
\end{definition}

We next mention the concept of semi-inner-products (s.i.p.) in Banach spaces, which is essential for characterizing the norm attainment set of a bounded bilinear operator between Banach spaces. 

\begin{definition}
Let $ \mathbb{X} $ be a normed space. A function $ [ ~,~ ] : \mathbb{X} \times \mathbb{X} \longrightarrow \mathbb{R} $ is a semi-inner-product (s.i.p.) if for any $ \alpha,~\beta \in \mathbb{R} $ and for any $ x,~y,~z \in \mathbb{X}, $ it satisfies the following:\\
$ (a) $ $ [\alpha x + \beta y, z] = \alpha [x,z] + \beta [y,z], $\\
$ (b) $ $ [x,x] > 0, $ whenever $ x \neq 0, $\\
$ (c) $ $ |[x,y]|^{2} \leq [x,x] [y,y], $\\
$ (d) $ $ [x,\alpha y] = \alpha [x,y]. $
\end{definition} 

It was proved in \cite{G} that every normed space $ (\mathbb{X},\|~\|) $ can be represented as a s.i.p. space $ (\mathbb{X},[~,~]) $ such that for all $ x \in \mathbb{X}, $ we have that $ [x,x] = \| x \|^{2}. $ In this article, whenever we speak of a s.i.p. $ [~,~] $ in context of a Banach space $ \mathbb{X} $, we implicitly assume that  $ [~,~] $ is compatible with the norm, i.e., for all $ x \in \mathbb{X}, $ we have that $ [x,x] = \| x \|^{2}. $ As an application of the concept of s.i.p. in Banach spaces, we obtain a complete characterization of the norm attainment set of a bounded bilinear operator between Banach spaces. We would like to further observe that our characterization is particularly useful when the underlying Banach spaces are smooth. Let us recall the basic definitions related to the concept of smoothness in Banach spaces. A Banach space $ \mathbb{X} $ is said to be smooth at a point $ x \in S_{\mathbb{X}} $ if there exists a unique supporting hyperplane to $ B_{\mathbb{X}} $ at $ x. $ Let $ \textrm{sm}~S_{\mathbb{X}} $ denote the set of all smooth points on the unit sphere $ S_{\mathbb{X}} $ of $ \mathbb{X}. $ We say that $ \mathbb{X} $ is smooth at a non-zero $ x \in \mathbb{X} $ if $ \mathbb{X} $ is smooth at $ \frac{x}{\|x\|} \in S_{\mathbb{X}}. $ Finally, $ \mathbb{X} $ is said to be a smooth Banach space if $ \mathbb{X} $ is smooth at each non-zero point of $ \mathbb{X}
. $

\section{Main Results}
Let $ \mathbb{X},\mathbb{Y} $ be Banach spaces. We begin this section by describing the positive part, the negative part and the Birkhoff-James orthogonality set of an element $ (x,y) \in \mathbb{X} \times \mathbb{Y}. $ As we will see, the answer depends on whether $ \| x \| = \| y \|. $

\begin{prop}\label{prop1}
Let $ \mathbb{X},\mathbb{Y} $ be Banach spaces and let $ (x,y) \in \mathbb{X} \times \mathbb{Y}, $ with $ \| x \| > \| y \|. $ Then the following are true:\\

(i) $ (x,y)^{+} = x^{+} \times \mathbb{Y}, $ \enspace (ii) $ (x,y)^{-} = x^{-} \times \mathbb{Y}, $ \enspace (iii) $ (x,y)^{\perp} = x^{\perp} \times \mathbb{Y}. $

\end{prop}

\begin{proof}
(i): Let $ (x_1,y_1) \in (x,y)^{+}. $ If possible, suppose that $ x_1 \notin x^{+}. $ Then there exists $ \lambda_0 > 0 $ such that $ \| x + \lambda x_1 \| < \| x \| $ for all $ \lambda \in (0, \lambda_0). $ Since $ \| x \| > \| y \|, $ it follows that for sufficiently small $ \lambda > 0, $ we have that $ \| (x,y) + \lambda (x_1,y_1) \| = \max \{ \| x + \lambda x_1 \|, \| y + \lambda y_1 \| \} = \| x+\lambda x_1 \| < \| x \| = \| (x,y) \|. $ This contradicts our assumption that $ (x_1,y_1) \in (x,y)^{+} $ and completes the proof of the fact that $ (x,y)^{+} \subseteq x^{+} \times \mathbb{Y}. $ Conversely, if $ (x_2,y_2) \in x^{+} \times \mathbb{Y} $ then for any $ \lambda \geq 0, $ we have that $ \| (x,y) + \lambda (x_2,y_2) \| \geq \| x + \lambda x_2 \| \geq \| x \| = \| (x,y) \|. $\\ This gives the set inclusion in the other direction  $ x^{+} \times \mathbb{Y} \subseteq (x,y)^{+} $ and completes the proof of our assertion.\\

(ii): The proof of (ii) can be completed similarly as above.\\

(iii): It follows from Proposition $2.1$ of \cite{S} and the above two results that $ (x,y)^{\perp} = (x,y)^{+} \bigcap (x,y)^{-} = (x^{+} \times \mathbb{Y}) \bigcap (x^{-} \times \mathbb{Y}) = x^{\perp} \times \mathbb{Y}. $
\end{proof}

On the other hand, if $ \| x \| = \| y \| $ in the above proposition, then we have the following result:

\begin{prop}\label{prop2}
Let $ \mathbb{X},\mathbb{Y} $ be Banach spaces and let $ (x,y) \in \mathbb{X} \times \mathbb{Y}, $ with $ \| x \| = \| y \|. $ Then the following are true:\\

(i) $ (x,y)^{+} = (x^{+} \times \mathbb{Y}) \bigcup (\mathbb{X} \times y^{+}), $ \enspace (ii) $ (x,y)^{-} = (x^{-} \times \mathbb{Y}) \bigcup (\mathbb{X} \times y^{-}), $\\
(iii) $ (x,y)^{\perp} = (x^{+} \times y^{-}) \bigcup (x^{-} \times y^{+}).  $

\end{prop}

\begin{proof}

(i): We first show that $ (x,y)^{+} \subseteq (x^{+} \times \mathbb{Y}) \bigcup (\mathbb{X} \times y^{+}). $ Let $ (x_1,y_1) \in (x,y)^{+}. $ If $ x_1 \in x^{+} $ then we are done. Suppose $ x_1 \notin x^{+}. $ Then there exists $ \lambda_1 > 0 $ such that $ \| x + \lambda x_1 \| < \| x \| $ for all $ \lambda \in (0, \lambda_1). $ We claim that $ y_1 \in y^{+}. $ If possible, suppose that our claim is not true. Then there exists $ \lambda_2 > 0 $ such that $ \| y + \lambda y_1 \| < \| y \| $ for all $ \lambda \in (0, \lambda_2). $ Therefore, for any $ \lambda \in (0, \min{\lambda_1,\lambda_2}), $ we have that $ \| (x,y) + \lambda (x_1),y_1 \| = \max\{\| x + \lambda x_1 \|, \| y + \lambda y_1 \|\} < \| x \| = \| y \| = \| (x,y) \|. $ However, this clearly contradicts our assumption that $ (x_1,y_1) \in (x,y)^{+}. $ This completes the proof of our claim. Therefore, we have that  $ (x,y)^{+} \subseteq (x^{+} \times \mathbb{Y}) \bigcup (\mathbb{X} \times y^{+}). $ We would like to note that the set inclusion in the other direction $ (x^{+} \times \mathbb{Y}) \bigcup (\mathbb{X} \times y^{+}) \subseteq (x,y)^{+} $ follows quite trivially, following the same line of arguments as given in the corresponding part of the proof of Proposition \ref{prop1}.\\

(ii): The proof of (ii) can be completed similarly as above.\\

(iii): It follows from Proposition $2.1$ of \cite{S} and the above two results that 

\begin{align*} 
(x,y)^{\perp} & = (x,y)^{+} \bigcap (x,y)^{-} \\
              & = \big[(x^{+} \times \mathbb{Y}) \bigcup (\mathbb{X} \times y^{+})\big]~ \bigcap~ \big[(x^{-} \times \mathbb{Y}) \bigcup (\mathbb{X} \times y^{-})\big] \\
              & = \big[\{(x^{+} \times \mathbb{Y}) \bigcup (\mathbb{X} \times y^{+})\} \bigcap (x^{-} \times \mathbb{Y})\big]~\bigcup~\big[\{(x^{+} \times \mathbb{Y}) \bigcup (\mathbb{X} \times y^{+})\} \bigcap (\mathbb{X} \times y^{-})\big] \\
              & = \big[(x^{\perp} \times \mathbb{Y}) \bigcup (x^{-} \times y^{+})\big]~\bigcup~ \big[(x^{+} \times y^{-}) \bigcup (\mathbb{X} \times y^{\perp})\big]\\
              & = (x^{+} \times y^{-}) \bigcup (x^{-} \times y^{+}).               
\end{align*}

We would like to note that in the above deduction, the last equality follows from the fact that $ (x^{\perp} \times \mathbb{Y})~\bigcup~ (\mathbb{X} \times y^{\perp}) \subseteq (x^{+} \times y^{-})~ \bigcup~ (x^{-} \times y^{+}), $ which can be proved by an easy application of Proposition $ 2.1 $ of \cite{S}.

\end{proof}

As an application of Proposition \ref{prop1}, we next describe the smooth points on the unit sphere of the Banach space $ \mathbb{X} \times \mathbb{Y}. $

\begin{theorem}\label{smooth1}
Let $ \mathbb{X},\mathbb{Y} $ be Banach spaces. Then $ \textrm{sm}~S_{\mathbb{X} \times \mathbb{Y}} = \big[(\textrm{sm}~S_{\mathbb{X}})
 \times (B_{\mathbb{Y}} \setminus S_{\mathbb{Y}})\big]~\bigcup~\big[(B_{\mathbb{X}} \setminus S_{\mathbb{X}}) \times (\textrm{sm}~S_{\mathbb{Y}})\big]. $
\end{theorem}

\begin{proof}
Let $ (x,y) \in \textrm{sm}~S_{\mathbb{X} \times \mathbb{Y}}. $ Clearly, $ \| (x,y) \| = \max\{ \|x\|,\|y\| \} = 1. $ Without any loss of generality, we assume that $ \| x \| = 1 \geq \| y \|. $ Our first claim is that $ \| y \| < 1. $ If possible, suppose that $ \|y\|=1. $ Now, it is easy to see that $ (x,y) \perp_B (-\frac{1}{2}x,y). $ Indeed, for any $ \lambda \geq 0, $ we have that $ \| (x,y) + \lambda (-\frac{1}{2}x,y) \| \geq \| (1+\lambda) y \| = 1 + \lambda \geq 1 = \| (x,y) \|. $ On the other hand, for any $ \lambda < 0, $ we have that $ \| (x,y) + \lambda (-\frac{1}{2}x,y) \| \geq \| (1 - \frac{1}{2}\lambda) x \| = 1 - \frac{1}{2} \lambda > 1 = \| (x,y) \|. $ Similarly, it is easy to see that $ (x,y) \perp_B (x,-\frac{1}{2}y). $ However, it is immediate that $ (x,y) \not\perp_B \frac{1}{2} (x,y). $ Since $ \frac{1}{2} (x,y) = (-\frac{1}{2}x,y) + (x,-\frac{1}{2}y), $ it follows that Birkhoff-James orthogonality in $ \mathbb{X} \times \mathbb{Y} $ is not right additive at the point $ (x,y). $ Therefore, applying the Theorem $ 4.2 $ of \cite{J}, we deduce that $ (x,y) $ is not a smooth point of $ \mathbb{X} \times \mathbb{Y}, $ a contradiction to our assumption. This completes the proof of our claim that $ \|y\| < 1, $ i.e., $ y \in B_{\mathbb{Y}} \setminus S_{\mathbb{Y}}. $ Our next claim is that $ x \in \textrm{sm}~S_{\mathbb{X}}. $ If possible, suppose that $ x \notin \textrm{sm}~S_{\mathbb{X}}. $ Then there exists $ u_1, u_2 \in \mathbb{X} $ such that $ x \perp_B u_1, x \perp_B u_2 $ but $ x \not\perp_B (u_1+u_2). $ Since $ \| y \| < 1, $ it follows from Proposition \ref{prop1} that $ (x,y) \perp_B (u_1,y), (x,y) \perp_B (u_2,y) $ but $ (x,y) \not\perp_B (u_1+u_2,y), $ once again contradicting our hypothesis that $ (x,y) \in \textrm{sm}~S_{\mathbb{X} \times \mathbb{Y}}. $ This completes the proof of our claim that $ x \in \textrm{sm}~S_{\mathbb{X}}. $ Moreover, it is now obvious that if we had assumed $ \| y \| = 1, $ instead of assuming $ \| x \| = 1, $ then we would have obtained the conclusion that $ \| x \| < 1 $ and $ y \in \textrm{sm}~S_{\mathbb{Y}}. $ This completes the proof of the fact that $ \textrm{sm}~S_{\mathbb{X} \times \mathbb{Y}} \subseteq \big[(\textrm{sm}~S_{\mathbb{X}})
 \times (B_{\mathbb{Y}} \setminus S_{\mathbb{Y}})\big]~\bigcup~\big[(B_{\mathbb{X}} \setminus S_{\mathbb{X}}) \times (\textrm{sm}~S_{\mathbb{Y}})\big]. $ Conversely, let us assume that $ (z,w) \in (\textrm{sm}~S_{\mathbb{X}}) \times (B_{\mathbb{Y}} \setminus S_{\mathbb{Y}}). $ Then once again applying Proposition $ 2.1, $ it is easy to see that Birkhoff-James orthogonality is right additive at the point $ (z,w). $ This completes the proof of the fact that $ (\textrm{sm}~S_{\mathbb{X}}) \times (B_{\mathbb{Y}} \setminus S_{\mathbb{Y}}) \subseteq \textrm{sm}~S_{\mathbb{X} \times \mathbb{Y}}. $ Similarly, it can be shown that $ (B_{\mathbb{X}} \setminus S_{\mathbb{X}}) \times (\textrm{sm}~S_{\mathbb{Y}}) \subseteq \textrm{sm}~S_{\mathbb{X} \times \mathbb{Y}}. $ Combining all these conclusions, we deduce that $ \textrm{sm}~S_{\mathbb{X} \times \mathbb{Y}} = \big[(\textrm{sm}~S_{\mathbb{X}})
 \times (B_{\mathbb{Y}} \setminus S_{\mathbb{Y}})\big]~\bigcup~\big[(B_{\mathbb{X}} \setminus S_{\mathbb{X}}) \times (\textrm{sm}~S_{\mathbb{Y}})\big]. $ This establishes the theorem.
\end{proof}

As a corollary to the above result, we point out a special fact about the norm attainment of bounded bilinear operators between Banach spaces.

\begin{cor} \label{cor1}
Let $ \mathbb{X}, \mathbb{Y}, \mathbb{Z} $ be Banach spaces and let $ \mathcal{T} \in \mathscr{B}(\mathbb{X} \times \mathbb{Y},\mathbb{Z}) $ be non-zero. Then $ M_{\mathcal{T}} \bigcap (\textrm{sm}~S_{\mathbb{X} \times \mathbb{Y}})  = \emptyset. $
\end{cor}

\begin{proof}
If $ \mathcal{T} $ does not attain norm, i.e., if $ M_{\mathcal{T}} = \emptyset $ then we have nothing more to show. Let us assume that $ (x,y) \in M_{\mathcal{T}}. $ Clearly, either $ \| x \| = 1 $ or $ \| y \| = 1. $ Without any loss of generality, we assume that $ \| x \| = 1. $ We claim that $ \| y \| = 1. $ If possible suppose that $ \| y \| < 1. $ Since $ \mathcal{T} $ is not the zero operator, it is easy to observe that $ \| y \| > 0. $ Now, $ \| (x, \frac{y}{\|y\|}) \| = 1 $ and it follows from the bilinearity of $ \mathcal{T} $ that $ \| \mathcal{T}(x, \frac{y}{\|y\|}) \| = \frac{1}{\|y\|} \| \mathcal{T}(x, y) \| = \frac{1}{\|y\|} \| \mathcal{T} \| > \| \mathcal{T} \|, $ contradicting the definition of $ \| \mathcal{T} \|. $ This contradiction completes the proof of our claim that $ \| y \| = 1. $ Therefore, it follows from Theorem \ref{smooth1} that $ (x,y) \notin \textrm{sm}~S_{\mathbb{X} \times \mathbb{Y}}.  $ Since this is true for any $ (x,y) \in M_{\mathcal{T}}, $ the proof of the theorem stands completed.
\end{proof}

In the next remark, we point out a fundamental difference in the norm attainment of bounded linear operators between Banach spaces and bounded bilinear operators between Banach spaces.

\begin{remark}
As seen in Corollary \ref{cor1}, a non-zero bounded bilinear operator between Banach spaces never attains norm at any smooth point of the unit sphere of the domain space. However, this is not necessarily true in case of the norm attainment of bounded linear operators between Banach spaces. Of course, there exists a bounded linear operator on a Banach space that does not attain norm at any smooth point of the space. However, given any Banach space, it is also easy to observe that the set of smooth points on the unit sphere of the space is always non-empty. Indeed, there exists a non-zero bounded linear operator on the space that attains norm at some smooth point of the unit sphere of the domain space.
\end{remark}

On the other hand, the study of Birkhoff-James orthogonality of bounded bilinear operators is analogous to the case of bounded linear operators. In the following theorem, we obtain a complete characterization of the Birkhoff-James orthogonality of bilinear operators on finite-dimensional Banach spaces. We would like to remark that we present only an outline of the proof of the next theorem, as it is analogous to the proof of Theorem $ 2.2 $ of \cite{S}.

\begin{theorem}\label{orthogonality}
Let $ \mathbb{X},\mathbb{Y}, \mathbb{Z} $ be finite-dimensional Banach spaces. Let $ \mathcal{T}, \mathcal{A} \in \mathscr{B}(\mathbb{X} \times \mathbb{Y}, \mathbb{Z}). $ Then $ \mathcal{T} \perp_B \mathcal{A} $ if and only if there exists $ (x_1,y_1), (x_2,y_2) \in M_{\mathcal{T}} $ such that $ \mathcal{A}(x_1,y_1) \in \mathcal{T}(x_1,y_1)^{+} $ and $ \mathcal{A}(x_2,y_2) \in \mathcal{T}(x_2,y_2)^{-}. $
\end{theorem}

\begin{proof}
The sufficient part of the theorem follows trivially. We prove the necessary part of the theorem by the method of contradiction. Since $ \mathbb{X} \times \mathbb{Y} $ is finite-dimensional, it is clear that $ M_{\mathcal{T}} \neq \emptyset. $ If possible, suppose that there does not exist any $ (x,y) \in M_{\mathcal{T}} $ such that $ \mathcal{A}(x,y) \in \mathcal{T}(x,y)^{+}. $ Then for any $ (x,y) \in M_{\mathcal{T}}, $ there exists $ \lambda_{(x,y)} > 0 $ such that $ \| \mathcal{T}(x,y) + \lambda_{(x,y)} \mathcal{A}(x,y) \| < \| \mathcal{T}(x,y) \| = \| \mathcal{T} \|.  $ Using the continuity and the convexity properties of the norm function, it is easy to show that there exist $ r_{(x,y)} > 0 $ such that $ \| \mathcal{T}(z,w) + \lambda \mathcal{A}(z,w) \| < \| \mathcal{T} \| $ for all $ (z,w) \in B((x,y),r_{(x,y)}) \bigcap S_{\mathbb{X} \times \mathbb{Y}} $ and for all $ \lambda \in (0, \lambda_{(x,y)}). $ On the other hand, given any $ (u,v) \in S_{\mathbb{X} \times \mathbb{Y}} \setminus M_{\mathcal{T}}, $ clearly there exists $ r_{(u,v)}, \delta_{(u,v)} > 0 $ such that $ \| \mathcal{T}(u_1,v_1) + \lambda \mathcal{A} (u_1,v_1) \| < \| \mathcal{T} \| $ for all $ (u_1,v_1) \in B((u,v),r_{(u,v)}) \bigcap S_{\mathbb{X} \times \mathbb{Y}} $ and for all $ \lambda \in (-\delta_{(u,v)},\delta_{(u,v)}). $ Now, $ \{ B((x,y),r_{(x,y)}) \bigcap S_{\mathbb{X} \times \mathbb{Y}} : (x,y) \in M_{\mathcal{T}} \}\\
~\bigcup~\{ B((u,v),r_{(u,v)}) \bigcap S_{\mathbb{X} \times \mathbb{Y}} : (u,v) \in S_{\mathbb{X} \times \mathbb{Y}} \setminus M_{\mathcal{T}} \} $ is an open cover of $ S_{\mathbb{X} \times \mathbb{Y}}. $ Since $ S_{\mathbb{X} \times \mathbb{Y}} $ is compact, this open cover has a finite subcover. Depending on this subcover, it is now easy to choose a suitable $ \lambda_0 > 0 $ such that $ \| \mathcal{T}(\alpha,\beta) + \lambda_0 \mathcal{A}(\alpha,\beta) \| < \| \mathcal{T} \| $ for any  $ (\alpha,\beta) \in S_{\mathbb{X} \times \mathbb{Y}}. $ Since $ \mathcal{T} + \lambda_0 \mathcal{A} $ attains norm, this clearly implies that $ \| \mathcal{T}+\lambda_0 \mathcal{A} \| < \| \mathcal{T} \|, $ a contradiction to our hypothesis that $ \mathcal{T} \perp_B \mathcal{A}. $ Therefore, there exists  $ (x_1,y_1) \in M_{\mathcal{T}} $ such that $ \mathcal{A}(x_1,y_1) \in \mathcal{T}(x_1,y_1)^{+}. $ Similarly, we can show that there exists  $ (x_2,y_2) \in M_{\mathcal{T}} $ such that $ \mathcal{A}(x_2,y_2) \in \mathcal{T}(x_2,y_2)^{-}. $ This establishes the theorem.
\end{proof}

We obtain the following corollary to the above theorem, that turns out to be useful in characterizing the smooth bilinear operators between finite-dimensional Banach spaces.

\begin{cor}\label{cor2}
Let $ \mathbb{X}, \mathbb{Y}, \mathbb{Z} $ be finite-dimensional Banach spaces and let $ \mathcal{T} \in \mathscr{B}(\mathbb{X} \times \mathbb{Y}, \mathbb{Z}) $ be such that $ M_{\mathcal{T}} = \{(\pm x_0, \pm y_0) \}. $ Then for any $ \mathcal{A} \in \mathscr{B}(\mathbb{X} \times \mathbb{Y}, \mathbb{Z}) , $ we have that $ \mathcal{T} \perp_B \mathcal{A} $ if and only if $ \mathcal{T}(x_0,y_0) \perp_B \mathcal{A}(x_0,y_0). $
\end{cor}

\begin{proof}
The sufficient part of the theorem follows trivially. Let us prove the necessary part of the corollary. It follows from Theorem \ref{orthogonality} that there exists $ (u_1,v_1), (u_2,v_2) \in M_{\mathcal{T}} $ such that $ \mathcal{A}(u_1,v_1) \in \mathcal{T}(u_1,v_1)^{+} $ and  $ \mathcal{A}(u_2,v_2) \in \mathcal{T}(u_2,v_2)^{-}. $ Without any loss of generality, we assume that $ \mathcal{A}(x_0,y_0) \in \mathcal{T}(x_0,y_0)^{+}. $ Now, if $ \mathcal{A}(-x_0,-y_0) \in \mathcal{T}(-x_0,-y_0)^{-} $ then it follows from the bilinearity of $ \mathcal{T} $ and $ \mathcal{A} $ that $ \mathcal{A}(x_0,y_0) \in \mathcal{T}(x_0,y_0)^{-} $ and therefore, it follows from Proposition $ 2.1 $ of \cite{S} that $ \mathcal{T}(x_0,y_0) \perp_B \mathcal{A}(x_0,y_0).  $ On the other hand, if $ \mathcal{A}(-x_0,y_0) \in \mathcal{T}(-x_0,y_0)^{-} $ or if $ \mathcal{A}(x_0,-y_0) \in \mathcal{T}(x_0,-y_0)^{-}, $ then by using similar arguments we can deduce that $ \mathcal{T}(x_0,y_0) \perp_B \mathcal{A}(x_0,y_0). $ This completes the proof of the corollary.
\end{proof}

We are now in a position to completely characterize the smooth bilinear operators between finite-dimensional Banach spaces. The desired characterization is obtained in the following theorem.

\begin{theorem}\label{smooth2}
Let $ \mathbb{X},\mathbb{Y}, \mathbb{Z} $ be finite-dimensional Banach spaces. Let $ \mathcal{T} \in \mathscr{B}(\mathbb{X} \times \mathbb{Y}, \mathbb{Z}) $ be non-zero. Then $ \mathcal{T} $ is a smooth point in $ \mathscr{B}(\mathbb{X} \times \mathbb{Y}, \mathbb{Z}) $ if and only if there exists $ (x_0,y_0) \in S_{\mathbb{X} \times \mathbb{Y}} $ such that $ M_{\mathcal{T}} = \{ (\pm x_0, \pm y_0) \} $ and $ \mathcal{T}(x_0,y_0) $ is a smooth point in $ \mathbb{Z}. $
\end{theorem}

\begin{proof}
Let us first prove the sufficient part of the theorem. Applying Theorem $ 4.2 $ of \cite{J}, it is enough to show that Birkhoff-James orthogonality is right additive at $ \mathcal{T}. $ Let $ \mathcal{A}_1, \mathcal{A}_2 \in \mathscr{B}(\mathbb{X} \times \mathbb{Y}, \mathbb{Z}) $ be such that $ \mathcal{T} \perp_B \mathcal{A}_1 $ and $ \mathcal{T} \perp_B \mathcal{A}_2. $ It follows from Corollary \ref{cor2} that $ \mathcal{T}(x_0,y_0) \perp_B \mathcal{A}_1(x_0,y_0) $ and $ \mathcal{T}(x_0,y_0) \perp_B \mathcal{A}_2(x_0,y_0). $ Since $ \mathcal{T}(x_0,y_0) $ is a smooth point in $ \mathbb{Z}, $ we have that $ \mathcal{T}(x_0,y_0) \perp_B (\mathcal{A}_1+\mathcal{A}_2)(x_0,y_0). $ Now, Theorem \ref{orthogonality} implies that $ \mathcal{T} \perp_B (\mathcal{A}_1+\mathcal{A}_2). $ This completes the proof of the sufficient part of the theorem. Let us now prove the necessary part of the theorem. Let $ \mathcal{T} $ be a smooth point in $ \mathscr{B}(\mathbb{X} \times \mathbb{Y}, \mathbb{Z}). $ If possible, suppose that $ M_{\mathcal{T}} \neq \{ (\pm x_0,\pm y_0) \} $ for any $ (x_0,y_0) \in S_{\mathbb{X} \times \mathbb{Y}}. $ Since $ M_{\mathcal{T}} \neq \emptyset,  $ this implies that there exist $ (x_1,y_1), (x_2,y_2) \in M_{\mathcal{T}} $ such that either $ x_2 \neq \pm x_1 $ or $ y_2 \neq \pm y_1. $ We note that it follows from the proof of Corollary \ref{cor1} that $ \|x_i\|=\| y_i \|=1, $ where $ i=1,2. $ Using this, it is easy to verify that $ \{ (x_1,y_1),(x_2,y_2) \} $ is linearly independent in $ \mathbb{X} \times \mathbb{Y}. $ We extend it to a basis $ \{ (x_1,y_1),(x_2,y_2),(v_i,w_i) : i=3,\ldots,n \} $ of $ \mathbb{X} \times \mathbb{Y}, $ where $ \dim (\mathbb{X} \times \mathbb{Y}) = n. $ It is now easy to define two bilinear operators $ \mathcal{A}_1,\mathcal{A}_2 \in \mathscr{B}(\mathbb{X} \times \mathbb{Y}, \mathbb{Z}) $ such that the following conditions are satisfied:\\

(i) $ \mathcal{T}(x_1,y_1) \perp_B \mathcal{A}_1(x_1,y_1), $ \enspace (ii) $ \mathcal{T}(x_2,y_2) \perp_B \mathcal{A}_2(x_2,y_2), $ \enspace (iii) $ \mathcal{T} = \mathcal{A}_1+\mathcal{A}_2. $\\

It follows from Theorem \ref{orthogonality} that $ \mathcal{T} \perp_B \mathcal{A}_1 $ and $ \mathcal{T} \perp_B \mathcal{A}_2. $ However, as $ \mathcal{T} $ is non-zero, we have that $ \mathcal{T} \not\perp_B (\mathcal{A}_1+\mathcal{A}_2). $ This contradicts our hypothesis that $ \mathcal{T} $ is  smooth in $ \mathscr{B}(\mathbb{X} \times \mathbb{Y}, \mathbb{Z}). $ This completes the proof of the fact that there exists $ (x_0,y_0) \in S_{\mathbb{X} \times \mathbb{Y}} $ such that $ M_{\mathcal{T}} = \{(\pm x_0,\pm y_0)\}. $ Our next objective is to show that $ \mathcal{T}(x_0,y_0) $ is a smooth point in $ \mathbb{Z}. $ Since $ \mathcal{T} $ is non-zero and $ (x_0,y_0) \in M_{\mathcal{T}}, $ it follows that $ \mathcal{T}(x_0,y_0) $ is non-zero. If possible, suppose that $ \mathcal{T}(x_0,y_0) $ is not a smooth point in $ \mathbb{Z}. $ Then there exists $ z_1, z_2 \in \mathbb{Z} $ such that $ \mathcal{T}(x_0,y_0) \perp_B z_1, $ $ \mathcal{T}(x_0,y_0) \perp_B z_2 $ but $ \mathcal{T}(x_0,y_0) \not\perp_B (z_1+z_2). $ Let us consider any $ \mathcal{A}_3,\mathcal{A}_4 \in \mathscr{B}(\mathbb{X} \times \mathbb{Y}, \mathbb{Z}) $ such that (i) $ \mathcal{A}_3(x_0,y_0) = z_1, $ \enspace (ii) $ \mathcal{A}_4(x_0,y_0) = z_2. $ It follows from Theorem \ref{orthogonality} and Corollary \ref{cor2} that $ \mathcal{T} \perp_B \mathcal{A}_3, \mathcal{T} \perp_B \mathcal{A}_4 $ but $ \mathcal{T} \not\perp_B (\mathcal{A}_1+\mathcal{A}_2), $ which contradicts the smoothness of $ \mathcal{T}. $ This proves that $ \mathcal{T}(x_0,y_0) $ must be a smooth point in $ \mathbb{Z} $ and establishes the theorem.
\end{proof}

Theorem \ref{smooth2} allows us to identify smooth bilinear operators between finite-dimensional Banach spaces. In the following example, we construct such a smooth bilinear operator in $ \mathscr{B}(\ell_{\infty}^{2} \times \ell_{\infty}^{2}, \ell_{\infty}^{2}), $ that serves as a ready application of theorem \ref{smooth2}. 

\begin{example} \label{ex1}
Let $ \mathcal{T}: \ell_{\infty}^{2} \times \ell_{\infty}^{2} \longrightarrow \ell_{\infty}^{2} $ be a bilinear operator given by the following conditions:\\

\noindent $ \mathcal{T}((1,1),(1,1))=(1,0), $ \\
$ \mathcal{T}((1,1),(1,-1))=\mathcal{T}((1,-1),(1,1))=\mathcal{T}((1,-1),(1,-1))=(0,0). $\\

It is easy to check that $ M_{\mathcal{T}} = \{ (\pm (1,1), \pm(1,1)) \}. $ Since $ \mathcal{T}((1,1),(1,1))=(1,0) $ is a smooth point in $ \ell_{\infty}^{2}, $ it follows from Theorem \ref{smooth2} that $ \mathcal{T} $ is smooth in $ \mathscr{B}(\ell_{\infty}^{2} \times \ell_{\infty}^{2}, \ell_{\infty}^{2}). $
\end{example}

We now study the norm attainment set of a bounded bilinear operator between Banach spaces. Recently, a complete characterization of the norm attainment set of a bounded linear operator between Banach spaces has been obtained in \cite{Sb}, by using the concept of semi-inner-product (s.i.p.) in Banach spaces. We illustrate in the next theorem that the same approach works in the case of bounded bilinear operators between Banach spaces as well.
 
\begin{theorem}\label{norm attainment}
Let $ \mathbb{X},\mathbb{Y}, \mathbb{Z} $ be  Banach spaces and Let $ \mathcal{T} \in \mathscr{B}(\mathbb{X} \times \mathbb{Y}, \mathbb{Z}). $  Let $ (x_0,y_0) \in S_{\mathbb{X} \times \mathbb{Y}}. $ Then $ (x_0,y_0) \in M_{\mathcal{T}} $ if and only if the following two conditions hold true:\\

(i) $ \| x \| = \| y \| = 1, $\\

(ii) there exist s.i.p. $ [ ~,~ ]_{\mathbb{X}},  [ ~,~ ]_{\mathbb{Y}} $ on $ \mathbb{X}, \mathbb{Y} $ respectively and s.i.p. $ [~,~]_1, [~,~]_2 $ on $ \mathbb{Z} $ such that the following holds true for every $ x \in \mathbb{X} $ and for every $ y \in \mathbb{Y}: $  
\[ [ \mathcal{T}(x_0,y), \mathcal{T}(x_0,y_0)]_{1} + [ \mathcal{T}(x,y_0), \mathcal{T}(x_0,y_0)]_{2} = \| \mathcal{T} \|^{2} ( [x,x_0]_{\mathbb{X}} + [y,y_0]_{\mathbb{Y}}). \]
\end{theorem}

\begin{proof}
The sufficient part of the theorem is trivially true. Indeed, just by putting $ x=x_0 $ and $ y=y_0, $ we obtain that $ 2 \| \mathcal{T}(x_0,y_0) \|^{2} = \| \mathcal{T} \|^{2} (\|x\|^2+\|y\|^2) = 2 \| \mathcal{T} \|^{2}. $ This is clearly equivalent to  $ (x_0,y_0) \in M_{\mathcal{T}}.  $ Let us now prove the necessary part of the theorem. We have already observed that since $ (x_0,y_0) \in M_{\mathcal{T}}, $ we must have that $ \| x_0 \|=\| y_0 \| = 1. $ In order to prove the next statement, let us define a bounded linear operator $ \mathcal{T}_{x_{0}}: \mathbb{Y} \longrightarrow \mathbb{Z} $ given by
\[ \mathcal{T}_{x_{0}} (y) = \mathcal{T} (x_0,y)~ \textrm{for each}~ y \in \mathbb{Y}. \]

It is easy to see that $ \| \mathcal{T}_{x_{0}} \| = \| \mathcal{T} \| $ and $ y_0 \in M_{\mathcal{T}_{x_{0}}} = \{ y \in S_{\mathbb{Y}} : \| \mathcal{T}_{x_{0}} (y) \| = \| \mathcal{T}_{x_{0}} \| \}. $ Therefore, by applying Theorem $ 2.2 $ of \cite{Sb}, we deduce that there exists two s.i.p. $ [~,~]_{\mathbb{Y}}, [~,~]_1 $ on $ \mathbb{Y} $ and $ \mathbb{Z} $ respectively such that $ [ \mathcal{T}_{x_{0}} (y), \mathcal{T}_{x_{0}} (y_0) ]_{1} = \| \mathcal{T}_{x_{0}} \|^{2} [ y,y_0 ]_{\mathbb{Y}} $ for each $ y \in \mathbb{Y}. $ However, this gives us that $ [ \mathcal{T}(x_0,y),\mathcal{T}(x_0,y_0) ]_{1} = \| \mathcal{T} \|^{2} [ y,y_0 ]_{\mathbb{Y}} $ for each $ y \in \mathbb{Y}. $ Similarly considering a bounded linear operator $ \mathcal{T}_{y_{0}}: \mathbb{X} \longrightarrow \mathbb{Z} $ given by $ \mathcal{T}_{y_{0}} (x) = \mathcal{T} (x,y_0)~ \textrm{for each}~ x \in \mathbb{X},  $ we deduce that there exists two s.i.p. $ [~,~]_{\mathbb{X}}, [~,~]_2 $ on $ \mathbb{X} $ and $ \mathbb{Z} $ respectively such that $ [ \mathcal{T}(x,y_0),\mathcal{T}(x_0,y_0) ]_{2} = \| \mathcal{T} \|^{2} [ x,x_0 ]_{\mathbb{X}} $ for each $ x \in \mathbb{X}.  $ The proof of the theorem can now be completed by adding these two relations. This establishes the theorem.
\end{proof}

Although theorem \ref{norm attainment} completely characterizes the norm attainment set of a bounded bilinear operator between Banach spaces, it is particularly useful when the concerned Banach spaces are smooth. This is because of the fact that in a smooth Banach space, there exists a unique s.i.p. on the space. It is therefore befitting that we end the present article with the following theorem whose proof is immediate from Theorem \ref{norm attainment} and the defining properties of s.i.p. 

\begin{theorem}\label{norm attainment2}
Let $ \mathbb{X},\mathbb{Y}, \mathbb{Z} $ be smooth Banach spaces and Let $ \mathcal{T} \in \mathscr{B}(\mathbb{X} \times \mathbb{Y}, \mathbb{Z}). $ Let $ [~,~]_{\mathbb{X}}, [~,~]_{\mathbb{Y}}, [~,~]_{\mathbb{Z}} $ be the unique s.i.p. on $ \mathbb{X}, \mathbb{Y}, \mathbb{Z} $ respectively. Let $ (x_0,y_0) \in S_{\mathbb{X} \times \mathbb{Y}}. $ Then $ (x_0,y_0) \in M_{\mathcal{T}} $ if and only if the following two conditions hold true:\\

(i) $ \| x \| = \| y \| = 1, $\\

(ii) for every $ x \in \mathbb{X} $ and for every $ y \in \mathbb{Y}: $  
\[ [ \mathcal{T}(x_0,y) + \mathcal{T}(x,y_0) , \mathcal{T}(x_0,y_0)]_{\mathbb{Z}}  = \| \mathcal{T} \|^{2} ( [x,x_0]_{\mathbb{X}} + [y,y_0]_{\mathbb{Y}}). \]
\end{theorem}

\bibliographystyle{amsplain}

\end{document}